\documentclass[11pt]{amsart}

\usepackage[utf8]{inputenc}

\usepackage[utf8]{inputenc}
\usepackage{amssymb}
\usepackage{graphicx}
\usepackage{graphicx,color}
\usepackage{latexsym}
\usepackage[all]{xy}
\usepackage{mathrsfs}
\usepackage{enumerate}
\usepackage{multicol}
\usepackage{lipsum}
\usepackage{amssymb}
\usepackage{tikz}
\usepackage{float}
\usepackage{bm}
\usepackage{mathptmx}
\usetikzlibrary{shapes.geometric}

\usepackage[backend=biber,style=alphabetic,doi=false,isbn=false,url=false]{biblatex}
\usepackage{amsmath}
\usepackage{amsfonts}
\usepackage{amssymb,enumerate}
\usepackage{amsthm}
\usepackage{fancyhdr}
\usepackage{hyperref}
\usepackage{cleveref}
\usepackage{xcolor}
\usepackage{enumitem}
\usepackage{float}
\usepackage{verbatim}
\usepackage{tikz}
\usetikzlibrary[patterns]
\usetikzlibrary{arrows}
\usepackage{mathrsfs}
\usetikzlibrary{arrows}
\usepackage{listings}
\theoremstyle{plain}
\newtheorem{proposition}{Proposition}
\newtheorem{theorem}{Theorem}

\newtheorem{corollary}{Corollary}

\newtheorem{prob}{Problem}
\newtheorem*{theorem*}{Theorem}

\theoremstyle{definition}

\newtheorem{ex}{Example}
\theoremstyle{remark}
\newtheorem{rem}[theorem]{Remark}

\newcommand{\ke}{\operatorname{Ker}}
\newcommand{\im}{\operatorname{Im}}

\numberwithin{equation}{section}
\begin{document}
	
	\title[Tjurina number for Sebastiani-Thom singularities]{The Tjurina number for Sebastiani-Thom type isolated hypersurface singularities}
	
	
	\author[P. Almir\'{o}n]{Patricio Almir\'{o}n}
	\address{Instituto de Matemática Interdisciplinar (IMI), Departamento de \'{A}lgebra, Geometr\'{i}a y Topolog\'{i}a\\
		Facultad de Ciencias Matem\'{a}ticas\\
		Universidad Complutense de Madrid\\
		28040, Madrid, Spain.
	}
	\email{palmiron@ucm.es}
	
	\thanks{The author is supported by Spanish Ministerio de Ciencia, Innovaci\'{o}n y Universidades MTM2016-76868-C2-1-P and by Spanish Ministerio de Ciencia, Innovaci\'{o}n y Universidades PID2020-114750GB-C32 }
	
	
	

	\begin{abstract}
	In this note we provide a formula for the Tjurina number of a join of isolated hypersurface singularities in separated variables. From this we are able to provide a characterization of isolated hypersurface singularities whose difference between the Milnor and Tjurina numbers is less or equal than two arising as the join of isolated hypersurface singularities in separated variables. Also, we are able to provide new upper bounds for the quotient of Milnor and Tjurina numbers of certain join of isolated hypersurface singularities. Finally, we deduce an upper bound for the quotient of Milnor and Tjurina numbers in terms of the singularity index of any isolated hypersurface singularity, not necessarily a join of singularities. 
		
	\end{abstract}
	\maketitle

	%
%

\section{Introduction}
Let \(f\in\mathbb{C}\{x_0,\dots,x_n\}\) be a germ of holmorphic function defining an isolated hypersurface singularity and let us denote by \(J_f=(\partial f/\partial x_0,\dots,\partial f/\partial x_n)\) the Jacobian ideal. The \textit{Tjurina number}

\[\tau_f:=\dim_{\mathbb{C}} \frac{\mathbb{C}\{x_0,\dots,x_n\}}{f+J_f}\]
is one of the most important analytic invariants of an hypersurface singularity. Being certainly not the finer analytic invariant of an isolated singularity, the deepening on its natural comparison with the \textit{Milnor number}, \(\mu_f:=\dim_{\mathbb{C}}\mathbb{C}\{x_0,\dots,x_n\}/J_f\), and its generalizations have provided several interesting results in Singularity Theory during the last 50 years.

In this note we will focus on the comparison of these two invariants in the particular case of the join of two isolated hypersurface singularities in separated variables, i.e. an isolated hypersurface singularity defined by \mbox{\(f+g\in\mathbb{C}\{\mathbf{x};\mathbf{y}\}=\mathbb{C}\{x_0,\dots,x_n;y_0,\dots,y_m\}\)} with  \(f\in\mathbb{C}\{\mathbf{x}\}\) and \(g\in\mathbb{C}\{\mathbf{y}\}.\) In this case, it is then easy to see that \(\mu_{f+g}=\mu_f\mu_g.\) Moreover, in 1971 Sebastiani and Thom \cite{sebastianithom} proved that the local monodromy of \(f+g\) is equivalent to the tensor product of the local monodromies of \(f\) and \(g.\) Thus, after 1971 it has been frequent in the literature to call a Sebastiani-Thom type isolated hypersurface singularity to a join of two isolated hypersurface singularities in separated variables. 

While the original Sebastiani-Thom Theorem is of topological nature, it is also natural to study up to what extent one can compute the main analytic invariants of a Sebastiani-Thom type singularity \(f+g\) from the analytic invariants of \(f\) and \(g.\) In this direction, the main results concern the mixed Hodge structure on the cohomology of the Milnor fiber \cite{SS85} and its related invariants, such as the spectrum of the singularity \cite{saito-distribution,Var81,SS85}. Also, there are several Sebastiani-Thom type theorems for other objects of analytic nature related to a hypersurface singularity such as multiplier ideals and Hodge ideals, see \cite{mustatasebmult,laurentiu}.

 From this point of view, with an eye on the nice expression of the Milnor number of a Sebastiani-Thom type singularity, it is natural to ask for an expression of \(\tau_{f+g}\) in terms of the Tjurina numbers \(\tau_f\) and \(\tau_g.\) Surprisingly, as far as the author knows, it does not exist any explicit result concerning a formula for the Tjurina number in the Sebastiani-Thom case. Our main result provides the following formula for the Tjurina number of a Sebastiani-Thom type singularity.

\begin{theorem}\label{thm:tjurina-st}
	Let \(f_1(\mathbf{x})\in\mathbb{C}\{\mathbf{x}\}=\mathbb{C}\{x_0,\dots,x_n\}\) and \(f_2(\mathbf{y})\in\mathbb{C}\{\mathbf{y}\}=\mathbb{C}\{y_0,\dots,y_n\}\) be germs of isolated hypersurface singularities in different variables. Then, we have
	\[\tau_{f_1+f_2}=\tau_{f_1}\tau_{f_2}+(\mu_{f_1}-\tau_{f_1})(\mu_{f_2}-\tau_{f_2})-\nu_1^{(f_2)}(\mu_{f_1}-\tau_{f_1}-\nu_1^{(f_1)})-\nu_1^{(f_1)}(\mu_{f_2}-\tau_{f_2}-\nu_1^{(f_2)})-b_{f_1+f_2}+u_{f_1+f_2}.\]
	In particular, 
	\[\tau_{f_1}\tau_{f_2}\leq \tau_{f_1+f_2}\leq \tau_{f_1}\tau_{f_2} +(\mu_{f_1}-\tau_{f_1})(\mu_{f_2}-\tau_{f_2}).\]
\end{theorem}
Here the numbers \(\nu_1^{(f_1)},\nu_1^{(f_2)},b_{f_1+f_2},u_{f_1+f_2}\) are nonnegative integers associated to certain dimensions of \(\mathbb{C}\)--vector spaces, which will be defined Section \ref{sec:tjurinast}. It is important to remark that the numbers \(\nu_1^{(f_1)},\nu_1^{(f_2)},b_{f_1+f_2},u_{f_1+f_2}\) are in general difficult to manage, moreover as we will see in the proof of Theorem \ref{thm:tjurina-st} the numbers \(b_{f_1+f_2},u_{f_1+f_2}\) seems to strongly depend on \(f_1+f_2\) rather than the separated \(f_1,f_2.\) 

In contrast with the formula, the bounds provided by Theorem \ref{thm:tjurina-st} are quite simple and also they are sharp. For that reason, we think that the important part of Theorem \ref{thm:tjurina-st} is precisely the lower and upper bounds; and more specifically the upper bound. That estimates for will allow us to characterize Sebastiani-Thom type singularities with \(\mu-\tau\leq 2\) (Corollary \ref{cor:charqh} and Corollary \ref{cor:smallmuminustau}). Also, from Theorem \ref{thm:tjurina-st} we can deduce that we cannot expect a \textquotedblleft nice\textquotedblright  expression for the Tjurina subspectrum for this family of singularities (see Section \ref{sec:tjurinast}).

The motivation to tackle the problem of finding a formula for the Tjurina number of a Sebastiani-Thom type singularity is to provide new particular cases of the following problem proposed by the author in \cite{quotpat}.
	\begin{prob}\cite[Problem 1]{quotpat}\label{problem}
		Let \((X,0)\subset(\mathbb{C}^N,0)\) be an isolated complete intersection singularity of dimension \(n\) and codimension \(k=N-n\). Is there an optimal \(\frac{b}{a}\in\mathbb{Q}\) with \(b<a\) such that 
		
		\[\mu-\tau<\frac{b}{a}\mu\;\text{?}\]
		Here optimal means that there exist a family of singularities such that \(\mu/\tau\)  tends to \(\frac{a}{a-b}\) when the multiplicity at the origin tends to infinity.
	\end{prob}

Problem \ref{problem} was originated as an extension of a question posed by Dimca and Greuel about the quotient of the Milnor and Tjurina numbers of a plane curve singularity \cite{dim}. In \cite{quotpat}, we showed the following solutions for \(r=1\): in the case of plane curve singularities we have \((a,b)=(1,4)\) and in the case of surface singularities satisfying Durfee's conjecture we have \((a,b)=(1,3).\) Combining those result together with Theorem \ref{thm:tjurina-st}, we are able to a solution to Problem \ref{problem} in the case of a join of surface singularities or plane curves with quasi-homogeneous functions (Proposition \ref{prop:newbounds2}). Also, in the case of the join of a plane curve singularity with a surface singularity (Proposition \ref{prop:newbounds1}) we will provide an upper bound for \(b/a\) (Proposition \ref{prop:newbounds1}). In another vein, we will use a result of Varchenko, in order to show an upper bound for \(\mu/\tau\) of an hypersurface singularity, not necessarily of Sebastiani-Thom type, in terms of the embedding dimension and the minimal spectral value (Proposition \ref{prop:tauspectral}).

The paper is organized as follows: in Section \ref{sec:tjurinast} we will prove the formula for the Tjurina number of a Sebastiani-Thom type singularity and we will show some its consequences. In Section \ref{sec:quotmutau}, we will study the quotient \(\mu/\tau\) for any isolated hypersurface singularity, not necessarily of Sebastiani-Thom type, we will present an upper bound in terms of the minimal spectral value (Proposition \ref{prop:tauspectral}). Finally, we will discuss the new cases of Problem \ref{problem} that can be obtained as a consequence of our formula for the Tjurina number of a Sebastiani-Thom type singularity.	
	
\subsection*{Acknowledgements}

I would like to thank to J.J. Moyano-Fernández, P.D. González-Pérez and J. Viu-Sos for the useful comments and suggestions during the preparation of this work. 

\section{Tjurina number in the Sebastiani-Thom case}\label{sec:tjurinast}
Let \(f\in\mathbb{C}\{x_0,\dots,x_n\}\) be a germ of isolated hypersurface singularity. Let us denote by 

\[M_f:=\frac{\mathbb{C}\{x_0,\dots,x_n\}}{(\frac{\partial f}{\partial x_0},\dots ,\frac{\partial f}{\partial x_n})},\quad T_f:=\frac{\mathbb{C}\{x_0,\dots,x_n\}}{(f,\frac{\partial f}{\partial x_0},\dots ,\frac{\partial f}{\partial x_n})}\]
the Milnor and Tjurina algebras respectively. Recall that, since multiplication by \(f\) in \(M_f\) is a \(\mathbb{C}\)--linear map, we can see the Tjurina algebra as the cokernel of this map, i.e. 

\[0\rightarrow \ke(f)\rightarrow M_f\xrightarrow{\cdot f}M_f\rightarrow T_f\rightarrow 0.\]

 In the case of a Sebastiani-Thom type singularity \(f_1+f_2\) with \(f_1(\mathbf{x})\in\mathbb{C}\{\mathbf{x}\}=\mathbb{C}\{x_0,\dots,x_n\}\) and \(f_2(\mathbf{y})\in\mathbb{C}\{\mathbf{y}\}=\mathbb{C}\{y_0,\dots,y_n\},\) the Milnor algebra \(M_{f_1+f_2}\) decomposes as the tensor product of the Milnor algebras \(M_{f_1}\) and \(M_{f_2}.\) Therefore, we have the following exact sequence

\[0\rightarrow \ke(f_1+f_2)\rightarrow M_{f_1}\otimes M_{f_2}\xrightarrow{f_1\otimes 1_{\tiny M_{f_2}}+1_{M_{f_1}}\otimes f_2} M_{f_1}\otimes M_{f_2}\rightarrow T_{f_1+f_2}\rightarrow 0.\]

In order to simplify notation, for two \(\mathbb{C}\)--vector spaces \(V\subset W\) we will denote \(W\setminus V\) to the complement of \(V\) in \(W,\) i.e \(V\oplus (W\setminus V)=W;\) which means that \(\{0\}= (W\setminus V)\cap V.\) Also, we will denote \(\dim:=\dim_{\mathbb{C}}\) since all vector spaces to be considered are over \(\mathbb{C}.\) 

Let us denote by \(A_i:=M_{f_i}\setminus(\ke(f_i)+\im(f_i))\) and \(B_i:=\im(f_i)\setminus(\im(f_i)\cap\ke(f_i)).\) Then we have the following decomposition of the Milnor algebra \mbox{\(M_{f_i}=\ke(f_i)\oplus B_{i}\oplus A_i.\)} Using the decomposition in direct sum of the Milnor algebras we are able to provide the proof of Theorem \ref{thm:tjurina-st}

%
%
%
\begin{proof}[Proof of Theorem \ref{thm:tjurina-st}]

In order to simplify notation let us denote by \(\mu_i:=\mu_{f_i}\) and \(\tau_{i}:=\tau_{f_i}.\) Also let us denote by 
\[\nu_1^{(f_i)}:=\dim (\ke(f_i)\cap\im(f_i)).\]
 Our aim is to describe \(\im(f_1+f_2).\) To do so observe that 

\begin{equation}\label{eqn:1}
	\begin{split}
		(f_1+f_2)(M_{f_1}\otimes M_{f_2})=&f_1(M_{f_1})\otimes M_{f_2}+M_{f_1}\otimes f_2(M_{f_2})\\=&f_1(\ke(f_1)\oplus B_1\oplus A_1)\otimes M_{f_2}+M_{f_1}\otimes f_2(\ke(f_2)\oplus B_2\oplus A_2)\\
	\end{split}
\end{equation}

where 

\begin{equation*}
	\begin{split}
		f_1(\ke(f_1)\oplus B_1\oplus A_1)\otimes M_{f_2}=(f_1(B_1)\otimes\ke(f_2))\oplus& (f_1(B_1)\otimes B_2)\oplus (f_1(B_1)\otimes A_2)\\
		\oplus&(f_1(A_1)\otimes\ke(f_2))\oplus (f_1(A_1)\otimes B_2)\oplus (f_1(A_1)\otimes A_2)
	\end{split}
\end{equation*}

and

\begin{equation*}
	\begin{split}
	M_{f_1}\otimes f_2(\ke(f_2)\oplus B_2\oplus A_2)=(\ke(f_1)\otimes f_2(B_2))\oplus& (B_1\otimes f_2( B_2))\oplus (A_1\otimes f_2(B_2))\\
		\oplus&(\ke(f_1)\otimes A_2)\oplus (B_1\otimes f_2(A_2))\oplus (A_1\otimes f_2(A_2)).
	\end{split}
\end{equation*}

Before to continue, observe that we have the following equality
\begin{equation}\label{eqn:U}
\begin{array}{ccc}
	U:=&f_1(B_1)\otimes\ke(f_2)\oplus f_1(A_1)\otimes \ke(f_2)+ \ke(f_1)\otimes f_2(B_2)\oplus\ke(f_1)\otimes f_1(A_1)&\\&\shortparallel&\\
	U':=&\ke(f_1)\otimes \im(f_2)+\im(f_1)\otimes \ke(f_2).&
\end{array}
\end{equation}
In order to check Equation \eqref{eqn:U}, we will show that the obvious inclusions \[
\begin{split}
(f_1(B_1)\otimes\ke(f_2))\oplus( f_1(A_1)\otimes \ke(f_2))\subset\im(f_1)\otimes \ke(f_2)&\quad \text{and}\\ (\ke(f_1)\otimes f_2(B_2))\oplus( \ke(f_1)\otimes f_2(A_2))\subset\ke(f_1)\otimes \im(f_2)&
\end{split}\]
are in fact equalities by a dimension argument. Since by definition \((B_1\oplus A_1)\cap \ke(f_1)=\{0\},\) then it is straightforward to check \[((B_1\oplus A_1)\otimes \ke(f_2))\cap\ke(f_1+f_2)=\{0\}.\] 
Thus, we have \(\dim(f_1(B_1\oplus A_1)\otimes \ke(f_2))=(\mu_1-\tau_1)\tau_2\) and in this way we obtain the following equality \[(f_1(B_1)\otimes\ke(f_2))\oplus( f_1(A_1)\otimes \ke(f_2))=\im(f_1)\otimes \ke(f_2).\]
 A similar argument shows that \(\ke(f_1)\otimes f_2(B_2)\oplus \ke(f_1)\otimes f_2(A_2)=\ke(f_1)\otimes \im(f_2).\) Therefore \(U=U'.\) Moreover, we can easily compute its dimension as
 \[\dim U=(\mu_1-\tau_1)\tau_2+(\mu_2-\tau_2)\tau_1-\nu_1^{(f_1)}\nu_1^{(f_2)}.\]
 since \(\dim((\ke(f_1)\cap\im(f_1))\otimes(\ke(f_2)\cap\im(f_2)))=\nu_1^{(f_1)}\nu_1^{(f_2)}.\)

We can now rewrite Equation \eqref{eqn:1} as follows

\begin{equation*}
	\begin{split}
		(f_1+f_2)(M_{f_1}\otimes M_{f_2})=& (f_1+f_2)(A_1\otimes A_2)+(f_1+f_2)(A_1\otimes B_2)\\+& (f_1+f_2)(B_1\otimes A_2)+U+ (f_1+f_2)(B_1\otimes B_2)
	\end{split}
\end{equation*}
At this point, we are going to show that the previous sum has the following decomposition as direct sum:
\begin{equation}\label{eqn:directsum}
	\begin{split}
		(f_1+f_2)(M_{f_1}\otimes M_{f_2})=& (f_1+f_2)(A_1\otimes A_2)\oplus(f_1+f_2)(A_1\otimes B_2)\oplus (f_1+f_2)(B_1\otimes A_2)\\
		\oplus&(U+ (f_1+f_2)(B_1\otimes B_2))
	\end{split}
\end{equation}

We start with \(U\cap (f_1+f_2)(A_1\otimes B_2)=\{0\},\) and similar arguments work for \(U\cap (f_1+f_2)(A_1\otimes A_2)=\{0\}\) and \(U\cap (f_1+f_2)(B_1\otimes A_2)=\{0\}.\) Assume that there exist \(0\neq u\in U\cap (f_1+f_2)(A_1\otimes B_2)\), then there are \(a_1\otimes b_2\in A_1\otimes B_2,\) \(\alpha\in\ke(f_1),\) \(\alpha'\in M_{f_1},\) \(\beta\in\ke(f_2),\) \(\beta'\in M_{f_2}\) such that

\[f_1(a_1)\otimes b_2+a_1\otimes f_2(b_2)=u=\alpha\otimes f_2(\beta')+f_1(\alpha')\otimes \beta.\]

Thus, \(a_1\otimes f_2(b_2)=\alpha\otimes f_2(\beta')+f_1(\alpha')\otimes \beta-f_1(a_1)\otimes b_2\) which implies that 
\[a_1\otimes f_2(b_2)\in ((\ke(f_1)+\im(f_1))\otimes M_{f_2})\cap (A_1\otimes \im(f_2))\] 
but \(((\ke(f_1)+\im(f_1))\otimes M_{f_2})\cap (A_1\otimes \im(f_2))=\{0\}\) and therefore it provides a contradiction.

Finally, we check \((f_1+f_2)(A_1\otimes A_2)\cap (f_1+f_2)(B_1\otimes A_2)=\{0\},\) and similar arguments work for the rest of intersections that remain to check. Assume that there exists \(0\neq u\in (f_1+f_2)(A_1\otimes A_2)\cap (f_1+f_2)(B_1\otimes A_2)\), then there are \(a_1\otimes a_2\in A_1\otimes A_2\) and \(b_1\otimes a'_2\in B_1\otimes A_2\) such that 

\[f_1(a_1)\otimes a_2+a_1\otimes f_2(a_2)=u=f_1(b_1)\otimes a'_2+b_1\otimes f_2(a'_2).\]
Thus, \(a_1\otimes f_2(a_2)= f_1(b_1)\otimes a'_2+b_1\otimes f_2(a'_2)-f_1(a_1)\otimes a_2\in \im(f_1)\otimes M_{f_2}\) which again provides a contradiction since \(A_1\cap \im(f_1)=\{0\}.\)
%
%

To finish, we only need to compute the dimension of each summand as follows

\begin{equation}\label{eqn:dimensions}
	\begin{split}
		\dim(f_1+f_2)(A_1\otimes A_2)=&\dim(A_1\otimes A_2)=\nu_1^{(f_1)}\nu_1^{(f_2)}\\
		\dim(f_1+f_2)(A_1\otimes B_2)=&\dim(A_1\otimes B_2)=\nu_1^{(f_1)}(\mu_2-\tau_2-\nu_1^{(f_2)})\\
			\dim(f_1+f_2)(B_1\otimes A_2)=&\dim(B_1\otimes A_2)=\nu_1^{(f_2)}(\mu_1-\tau_1-\nu_1^{(f_1)})\\
			\dim\; U=&(\mu_1-\tau_1)\tau_2+(\mu_2-\tau_2)\tau_1-\dim(\ke(f_1)\cap\im(f_1)\otimes\ke(f_2)\cap\im(f_2))
	\end{split}
\end{equation}

Let us denote by \(b_{f_1+f_2}:=\dim((f_1+f_2)(B_1\otimes B_2))\) and by \(u_{f_1+f_2}:=\dim(U\cap (f_1+f_2(B_1\otimes B_2)))\). Then combining Equation \eqref{eqn:directsum} and \eqref{eqn:dimensions} we have

\begin{equation*}
	\begin{split}
		\dim\im(f_1+f_2)=& \nu_1^{(f_1)}\nu_1^{(f_2)}+\nu_1^{(f_1)}(\mu_2-\tau_2-\nu_1^{(f_2)})+\nu_1^{(f_2)}(\mu_1-\tau_1-\nu_1^{(f_1)})+\dim U+ b_{f_1+f_2}-u_{f_1+f_2}\\
		=&(\mu_1-\tau_1)\tau_2+(\mu_2-\tau_2)\tau_1+\nu_1^{(f_1)}(\mu_2-\tau_2)+\nu_1^{(f_2)}(\mu_1-\tau_1)-2\nu_1^{(f_2)}\nu_1^{(f_1)}+ b_{f_1+f_2}-u_{f_1+f_2}.
	\end{split}
\end{equation*}

Then,
\begin{equation}\label{eqn:lasteq}
	\begin{split}
		\tau_{f_1+f_2}=&\dim\operatorname{Coker}(f_1+f_2)=\mu_{f_1+f_2}-\dim \im(f_1+f_2)\\
		=&\tau_1\tau_2+(\mu_1-\tau_1)\tau_2+(\mu_2-\tau_2)\tau_1+(\mu_1-\tau_1)(\mu_2-\tau_2)-\dim \im(f_1+f_2)\\
		=&\tau_1\tau_2+(\mu_1-\tau_1)(\mu_2-\tau_2)-\nu_1^{(f_2)}(\mu_1-\tau_1)-\nu_1^{(f_1)}(\mu_2-\tau_2)+2\nu_1^{(f_2)}\nu_1^{(f_1)}-b_{f_1+f_2}+u_{f_1+f_2}.
	\end{split}
\end{equation}

Let us now show the \textquotedblleft in particular \textquotedblright statement. First, observe that \[b_{f_1+f_2} \leq\dim(B_1\otimes B_2)=(\mu_1-\tau_1-\nu_1^{(f_1)})(\mu_2-\tau_2-\nu_1^{(f_2)}).\]

Then, 
\[\tau_{f_1+f_2}\geq\tau_1\tau_2+\nu_1^{(f_2)}\nu_1^{(f_1)}+u_{f_1+f_2}\geq\tau_1\tau_2,\]
where the last inequality follows from the fact that \(\nu_1^{(f_1)}\geq 0\) and \(u_{f_1+f_2}\geq 0.\)

In order to show the upper bound, we recall that \(\nu_1^{(f_i)}\leq(\mu_i-\tau_i)\) and \(u_{f_1+f_2}\leq b_{f_1+f_2}.\) Therefore,

\[\tau_{f_1+f_2}\leq \tau_1\tau_2+(\mu_1-\tau_1)(\mu_2-\tau_2)+\nu_1^{(f_2)}(\nu_1^{(f_1)}-(\mu_1-\tau_1))+\nu_1^{(f_1)}(\nu_1^{(f_2)}-(\mu_2-\tau_2))\]
from which the desire bound follows.
\end{proof}

\begin{rem}
	Our proof provides not only a formula for the Tjurina number for the Sebastiani-Thom type singularities but also provides an explicit expression of the structure of \(\im(f_1+f_2).\) In particular, by means of the good basis considered by Scherk \cite{scherkmanin}, Saito \cite{msaito-brieskorn} and Hertling \cite{hertling1,hertling2} one can describe a basis of \(\im(f_1+f_2)\) which behaves well with respect to the \(V\)--filtration associated to \(J_{f_1+f_2}+(f_1+f_2).\)
		\end{rem} 
	
	\begin{rem}
	A point of interest in this setting is the Tjurina subspectrum. In \cite{tjurinasubs}, Jung et al defined the Tjurina subspectrum \(\{\alpha_i^{Tj}\}_{i\in A}\) of \(f\). This subspectrum is a subset of the spectrum of \(f\) such that \(|A|=\tau\) where the indexes are chosen carefully with respect to certain properties of the Hodge ideal \(V\)--filtration associated to \(J_f+f.\) Recall that if \(\{\alpha_1,\dots,\alpha_{\mu_f}\}\) and \(\{\beta_1,\dots,\beta_{\mu_g}\}\) are the spectral numbers of \(f\) and \(g\) respectively then \[\{\alpha_i+\beta_k|\;1\leq i\leq \mu_f,\,1\leq k\leq \mu_g\}\] are the spectral numbers of \(f+g.\) However, Theorem \ref{thm:tjurina-st} shows that the Tjurina subspectrum does not behave so well with respect to the Sebastiani-Thom property since otherwise \(\tau_{f+g}\) would be the product of \(\tau_f\) and \(\tau_g\). Being obvious one of the implication, we wonder whether the Sebastiani-Thom property on the Tjurina subspectrum is equivalent to the fact that \(\tau_{f+g}=\tau_f\tau_g;\) where by the Sebastiani-Thom property on the Tjurina subspectrum we means that if \(\{\alpha_i^{Tj}\}_{i\in A_f}\) and \(\{\beta_k^{Tj}\}_{i\in A_g}\) are the Tjurina subspectrum of \(f\) and \(g\) respectively then \[\{\alpha_i^{Tj}+\beta^{Tj}_k|\;1\leq i\leq \tau_f,\,1\leq k\leq \tau_g\}\] is the Tjurina spectrum of \(f+g.\)
\end{rem}

A first consequence of Theorem \ref{thm:tjurina-st} is the following characterization of quasi-homogeneous singularities arising as the join of two functions.

\begin{corollary}\label{cor:charqh}
	Let \(F=f+g\in\mathbb{C}\{x_0,\dots,x_n;y_0,\dots,y_m\}\) be a germ of isolated hypersurface singularity defined as the join of \(f(\mathbf{x})\in\mathbb{C}\{\mathbf{x}\}=\mathbb{C}\{x_0,\dots,x_n\}\) and \(g(\mathbf{y})\in\mathbb{C}\{\mathbf{y}\}=\mathbb{C}\{y_0,\dots,y_n\}\). Then, the following are equivalent
	\begin{enumerate}
		\item \(F\) is quasi-homogeneous,
		\item \(f\) and \(g\) are quasi-homogeneous.
		\item \(\tau_{F}=\mu_f\mu_g.\)
	\end{enumerate}
	
\end{corollary}
\begin{proof}
	\((1)\Rightarrow(2)\), we assume \(F\) to be a quasi-homogeneous function. Let us suppose \(f\) being quasi-homogeneous and \(g\) not quasi-homogeneous. By Theorem \ref{thm:tjurina-st} and K. Saito's Theorem \cite{saitohom} we have 
	\[\tau_F=\tau_f\mu_g\lvertneqq\mu_f\mu_g=\mu_F.\]
	
	Now assume \(F\) to be a quasi-homogeneous function and \(f\) and \(g\) not quasi-homogeneous.  By Theorem \ref{thm:tjurina-st}
	
	\[\tau_F\leq\tau_f\tau_g+(\mu_f-\tau_f)(\mu_g-\tau_g)=\mu_f\mu_g+ (\tau_f-\mu_f)\tau_g+(\tau_g-\mu_g)\tau_f\lvertneqq\mu_f\mu_g,\]
	where the last inequality is strict because by Saito's Theorem \cite{saitohom} \((\tau_f-\mu_f)\tau_g+(\tau_g-\mu_g)\tau_f\lvertneqq 0.\) Therefore, in both cases we have a contradiction since by Saito's Theorem \(\mu_F=\tau_F.\)
	
	\((2)\Rightarrow(3)\) is a direct consequence of Theorem \ref{thm:tjurina-st} ans Saito's Theorem \cite{saitohom}.
	
	\((3)\Leftrightarrow(1)\) follows from Saito's Theorem \cite{saitohom}.
	
\end{proof}

It is also interesting the case where one of the singularities has \(e^{BS}(f)=2,\) since in this case \(\nu_1=\mu-\tau.\) A particular example of this situation is a join where one of the functions defines a plane curve singularity. In those cases, we have the following simplified expression of the Tjurina number.

\begin{corollary}\label{cor:maximaltau}
	Let \(f_1(\mathbf{x})\in\mathbb{C}\{\mathbf{x}\}=\mathbb{C}\{x_0,\dots,x_n\}\) and \(f_2(\mathbf{y})\in\mathbb{C}\{\mathbf{y}\}=\mathbb{C}\{y_0,\dots,y_n\}\) be germs of isolated hypersurface singularities in separated variables. Assume that \(e^{(BS)(f_1)}=2\) or \(e^{(BS)(f_2)}=2.\) Then, 
	\[\tau_{f_1+f_2}=\tau_{f_1}\tau_{f_2}+\nu_1^{(f_1)}\nu_1^{(f_2)}.\]
\end{corollary}
\begin{proof}
	 Assume that \(e^{(BS)(f_1)}=2,\) and the case \(e^{(BS)(f_2)}=2\) follows from similar arguments. Observe that if \(e^{(BS)(f_1)}=2\) then \(\nu_1^{(f_1)}=\mu_{f_1}-\tau_{f_1},\) which in particular implies \(B_1=\im(f_1)\setminus(\im(f_1)\cap\ke(f_1))=\{0\}.\) Thus, the claim follows from Theorem \ref{thm:tjurina-st}.
\end{proof}

\begin{ex}
	Let us consider the isolated hypersurface singularity in \(\mathbb{C}\{x_1,x_2,x_3,x_4,x_5,x_6\}\) defined by 
	\[f(x_1,x_2,x_3,x_4,x_5,x_6)=x_2^4-x_1^5+x_1^3x_2^2+x_4^4-x_3^5+x_3^3x_4^2+x_6^4-x_5^5+x_5^3x_6^2,\]
	which is the join of three irreducible plane curve singularities of the form \(G(x,y)=y^4-x^5+x^3y^2.\)
	
	We can compute the Milnor and Tjurina numbers of \(G,\) \(\mu(G)=12\) and \(\tau(G)=11.\) By Corollary \ref{cor:maximaltau}, if we consider the hypersurface in \(\mathbb{C}\{x_1,x_2,x_3,x_4\}\) defined by \(H=G(x_1,x_2)+G(x_3,x_4),\) we have that \[\tau(G(x_1,x_2)+G(x_3,x_4))=11^2+1=122.\] With the help of SINGULAR \cite{singular}, we can compute \(\nu_1^{(H)}=21.\) Therefore, by Corollary \ref{cor:maximaltau} we can compute 
	\[\tau_f=\tau_H\tau_G+\nu_1^{(H)}\nu_1^{(G)}=122\cdot11+21\cdot1=1363.\]
\end{ex}
To finish, we can provide a characterization of Sebastiani-Thom type singularities with \(\mu-\tau\leq 2.\)
\begin{corollary}\label{cor:smallmuminustau}
		Let \(f_1(\mathbf{x})\in\mathbb{C}\{\mathbf{x}\}=\mathbb{C}\{x_0,\dots,x_n\}\) and \(f_2(\mathbf{y})\in\mathbb{C}\{\mathbf{y}\}=\mathbb{C}\{y_0,\dots,y_n\}\) be germs of isolated hypersurface singularities in separated variables. Then,
		\begin{enumerate}
			\item \(\tau_{f_1+f_2}=\mu_{f_1+f_2}-1\) if and only if \(\mu_{f_1}=\tau_{f_1}=1\) and \(\mu_{f_2}-1=\tau_{f_2}\) or  \(\mu_{f_2}=\tau_{f_2}=1\) and \(\mu_{f_1}-1=\tau_{f_1}.\)
				\item \(\tau_{f_1+f_2}=\mu_{f_1+f_2}-2\) if and only if  \(f_1,f_2\) are in one of the following cases:
				\begin{enumerate}
					\item \(\mu_{f_1}=\tau_{f_1}=1\)  and \(\mu_{f_2}-2=\tau_{f_2}\). 
					\item \(\mu_{f_2}=\tau_{f_2}=1\) and \(\mu_{f_1}-2=\tau_{f_1}.\)
					\item \(\mu_{f_1}=\mu_{f_2}=2\) and \(\tau_{f_1}=\tau_{f_2}=1.\)
					\item \(\mu_{f_1}=\tau_{f_1}=2\) and \(\mu_{f_2}-1=\tau_{f_2}.\)
					\item \(\mu_{f_2}=\tau_{f_2}=2\) and \(\mu_{f_1}-1=\tau_{f_1}.\)
				\end{enumerate}
				  
		\end{enumerate}
\end{corollary}
\begin{proof}
	Let us start with (1). By Theorem \ref{thm:tjurina-st}, we have the following inequality
	\[\tau_{f_1}\tau_{f_2}+(\mu_{f_1}-\tau_{f_2})(\mu_{f_2}-\tau_{f_2})\geq \mu_{f_1+f_2}-1=\mu_{f_1}\mu_{f_2}-1.\]
	Therefore, we have 
	\[0\geq \tau_{f_1}(\tau_{f_2}-\mu_{f_2})+\tau_{f_2}(\tau_{f_1}-\mu_{f_1})\geq -1.\]
	Since \(\tau_{f_i}(\tau_{f_i}-\mu_{f_i})\) are integers numbers, then the inequality is satisfied if at least one of the \(f_i\) is a quasi-homogeneous singularity. Observe that by Corollary \ref{cor:charqh} both cannot be quasi-homogeneous since we are assuming \(\tau_{f_1+f_2}\neq \mu_{f_1+f_2}.\) Also, at least one of the \(f_i\) must be quasi-homogeneous since otherwise \(\tau_{f_1}(\tau_{f_2}-\mu_{f_2})+\tau_{f_2}(\tau_{f_1}-\mu_{f_1})\leq -2.\) Let us assume \(f_1\) is quasi-homogeneous, the case \(f_2\) will follow \textit{mutatis mutandis}. In that case, since \(\mu_{f_2}\neq\tau_{f_2}\) we have 
	\[\tau_{f_1}(\tau_{f_2}-\mu_{f_2})=-1\]
	which implies \(\tau_{f_1}=1=\tau_{f_2}-\mu_{f_2}.\) The converse implication is obviously trivial.
	
	Let us move to the case (2). As before, by Theorem \ref{thm:tjurina-st}, we have the following inequality
	
	\[0\geq \tau_{f_1}(\tau_{f_2}-\mu_{f_2})+\tau_{f_2}(\tau_{f_1}-\mu_{f_1})\geq -2.\]
	As in the proof of part (1) we have the following subcases: \(\tau_{f_1}(\tau_{f_2}-\mu_{f_2})=-2\) and \(\tau_{f_1}=\mu_{f_1},\) \(\tau_{f_2}(\tau_{f_1}-\mu_{f_1})=-2\) and \(\tau_{f_2}=\mu_{f_2},\) \(\tau_{f_2}(\tau_{f_1}-\mu_{f_1})=\tau_{f_1}(\tau_{f_2}-\mu_{f_2})=-1\). The claim now easily follows from the casuistic of each subcase. 
\end{proof}
%
%
%
%
%

\section{The quotient of Milnor and Tjurina numbers}\label{sec:quotmutau}

In 2017, Liu \cite{liu} showed that for any isolated hypersurface singularity defined by \(f:\mathbb{C}^{n+1}\rightarrow\mathbb{C}\) one has \(\mu/\tau\leq n+1\). However, following his proof we are going to show that we can actually be more precise about this upper bound. Recall that the Briançon-Skoda exponent of \(f\) is defined by 

\[e^{BS}(f):=\min\{k\in\mathbb{N}\;|\;f^k\in (\partial f/\partial x_0,\dots,\partial f/\partial x_n)\}.\]

According to the Brian\c{c}on-Skoda Theorem \cite{brianon}, we know that \(e^{BS}(f)\leq n+1.\) However, one can slightly improve that bound by using the spectrum of \(f.\) Recall that the spectrum is a discrete invariant formed by $\mu$ rational \emph{spectral numbers} (see \cite[\nopp II.8.1]{Kul98})
\[
\alpha_1,\dots,\alpha_\mu\in\mathbb{Q}\cap(0,n+1).
\]
They are certain logarithms of the eigenvalues of the monodromy on the middle cohomology of the Milnor fibre which correspond to the equivariant Hodge numbers of Steenbrink's mixed Hodge structure. From this set of invariants Varchenko \cite[Theorem 1.3]{Var81} proved the following result.

\begin{theorem}\label{thm:varchenkoexponent}
	Let \(f\in\mathbb{C}\{x_0,\dots,x_n\}\) be a germ of holmorphic function defining an isolated hypersurface singularity. Let \(\alpha_{min}\) be the minimal spectral number, then 
	\[e^{BS}(f)\leq \lfloor n+1-2\alpha_{min}\rfloor+1\]
\end{theorem}

\begin{rem}
	Our definition of spectral numbers follows K. Saito and M. Saito's definition \cite{saito-distribution,Sai83} in contrast with Steenbrink and Varchenko's definition \cite{Ste77,Var81}. This means that \(\alpha'\) is a spectral number with Steenbrink and Varchenko's definition if and only if \(\alpha'+1\) is a spectral number with Saito's definition.
\end{rem}

Therefore, Varchenko's Theorem \ref{thm:varchenkoexponent} allows to slightly improve the general upper bound provided by Liu in \cite{liu}.

\begin{proposition}\label{prop:tauspectral}
	Let \(f\) be a holomorphic function in \(\mathbb{C}\{x_0,\dots,x_n\}\). Let \(\alpha_{min}\) be the minimal spectral number of \(f\). Then, 
	\[\frac{\mu}{\tau}<e^{BS}(f).\]
	In particular, \(\mu/\tau<\lfloor n+1-2\alpha_{min}\rfloor+1.\)
\end{proposition}
\begin{proof}	
	If we denote by \((f^{i})\) the ideal of \(M_f\) generated by \((f^{i}),\) those ideals define a decreasing filtration
	
	\[(0)=(f^{e^{BS}(f)})\subset\cdots\subset(f^2)\subset(f)\subset M_f.\]
	Then, one can consider the following long exact sequence: 
	\[0\rightarrow\ke(f)\cap(f^{i})\rightarrow(f^{i})\xrightarrow{f}(f^i)\rightarrow(f^i)/(f^{i+1})\rightarrow 0\]
	
	where the middle map is the multiplication by \(f\). Then,
	\[\operatorname{dim}_\mathbb{C}\Big\{\frac{(f^i)}{(f^{i+1})}\Big\}=\dim_{\mathbb{C}} \ke(f)\cap(f^{i})\leq \tau.\]
	
	Therefore, 
	
	\[\mu=\operatorname{dim}_\mathbb{C}M_f=\operatorname{dim}_\mathbb{C} T_f+\sum_{i=1}^{e^{BS}(f)-1}\operatorname{dim}_\mathbb{C}\Big\{\frac{(f^i)M_f}{(f^{i+1})M_f}\Big\}\leq e^{BS}(f)\tau.\]
	Applying Varchenko's Theorem \ref{thm:varchenkoexponent} we obtain \(\mu/\tau<\lfloor n+1-2\alpha_{min}\rfloor+1.\)
\end{proof}

\begin{rem}
	Since \(\alpha_{min}>0\) then \(n+1-2\alpha_{min}<n+1\) from which we have \(\lfloor n+1-2\alpha_{min}\rfloor\leq n.\) Therefore, in the worst case we have Liu's result. In the case where \(\alpha_{min}>1/2,\) as for example rational singularities, then \(\lfloor n+1-2\alpha_{min}\rfloor\leq n-1\) and we obtain a strictly better upper bound than the one coming from Brian\c{c}on-Skoda Theorem. 
\end{rem}

After Proposition \ref{prop:tauspectral}, one can then estimate the \(b/a\) of Problem \ref{problem}. Unfortunately, it is easy to check that the bound provided by Proposition \ref{prop:tauspectral} is not sharp. In \cite{quotpat}, we showed that for any plane curve singularity \(\mu/\tau<4/3\) and moreover it is asymptotically sharp. This result provided a full answer to a question posed by Dimca and Greuel in \cite{dim}. The techniques used to prove that bound were based on the theory of surface singularities. More concretely, we showed the relation of Problem \ref{problem} with the long standing Durfee's conjecture \cite{durfee} and its generalization \cite{saito-distribution}, which claims that if \(f\in\mathbb{C}\{x_0,\dots,x_n\}\) defines an isolated hypersurface singularity then \((n+1)!p_g<\mu,\) where \(p_g\) is the geometric genus. With the help of Durfee's conjecture we were also able to show an asymptotically sharp upper bound in the case of surface singularities in \(\mathbb{C}^{3}.\) Those results are collected in the following
\begin{proposition}\cite[Thm. 6 and Prop. 3]{quotpat}\label{prop:oldbounds}
	\begin{enumerate}
		\item If \(f(x,y)\) is a plane curve then \(\mu/\tau<4/3.\)
		\item If \(f(x,y,z)\) is a surface singularity satisfying Durfee's conjecture then \(\mu/\tau<3/2.\)
	\end{enumerate}
\end{proposition}

Therefore, the combination of Theorem \ref{thm:tjurina-st} and Proposition \ref{prop:oldbounds} allow us to generalize our previous results to the following Sebastiani-Thom type singularities.
\begin{proposition}\label{prop:newbounds2}
	Let \(g(z_0,\dots,z_n)\) be a quasi-homogeneous function. Then,
	\begin{enumerate}
		\item If \(f(x,y)\) is an isolated plane curve singularity, then 
		\[\frac{\mu_{f+g}}{\tau_{f+g}}<\frac{4}{3}.\]
	
		\item If \(f(x_1,x_2,x_3)\) is an isolated surface singularity in \(\mathbb{C}^3\) satisfying Durfee's conjecture, then
		\[\frac{\mu_{f+g}}{\tau_{f+g}}<\frac{3}{2}.\]
		
	\end{enumerate}
 
\end{proposition}
\begin{proof}
	Since \(g\) is quasi-homogeneous then \(\tau_g=\mu_g\) and \(\tau_{f+g}=\tau_f\mu_g.\) Then in both cases we have \(\mu_{f+g}/\tau_{f+g}=\mu_f/\tau_f.\) Therefore, the claim follows from Proposition \ref{prop:oldbounds}.
\end{proof}
Moreover, it is easy to find families for which the bounds of Proposition \ref{prop:newbounds2} are asymptotically sharp. In the case of plane curve singularities, consider \(f(x,y)=y^{n}-x^{n+1}+g(x,y)\) with \( \deg_w(f) < \deg_w(g) \) with respect to the weights \(w=(n,n+1).\) Moreover, choose \(g\) such that \(\tau_f=\tau_{min},\) i.e. the Tjurina number of \(f\) is minimal over all possible Tjurina numbers in a \(\mu\)--constant deformation of \(y^{n}-x^{n+1}.\) Then \cite{taumin}, 
\[\tau_{min}=\frac{3 n^2}{4}-1 \quad \textnormal{if} \quad n\ \textnormal{is even},\quad
\tau_{min}=\frac{3}{4}(n^2-1) \quad \textnormal{if} \quad n\ \textnormal{is odd}.\]
Therefore the join of \(f\) with any quasi-homogeneous function \(h\) in separated variables gives \[\lim_{n\rightarrow \infty}\mu_{f+h}/\tau_{f+h}=4/3.\]
In \cite[Example 3]{quotpat}, we showed a family of surface singularities in \(\mathbb{C}^{3}\) with \(\mu/\tau\rightarrow 3/2.\) Using that example, the same reasoning as before allows to construct an example where the bound of Proposition \ref{prop:newbounds2} (2) is asymptotically sharp.

To finish, the following Proposition also follows from a direct application of Theorem \ref{thm:tjurina-st} and Proposition \ref{prop:oldbounds}.
\begin{proposition}\label{prop:newbounds1}
	Let \(f(x_1,x_2,x_3)\) be an isolated surface singularity  in \(\mathbb{C}^3\) satisfying Durfee's conjecture and \(g(z_1,z_2)\) be a plane curve singularity. Then 
	\[\frac{\mu_{f+g}}{\tau_{f+g}}<2.\]
\end{proposition}
\begin{proof}
	By Theorem \ref{thm:tjurina-st} we have \(\tau_{f+g}\geq \tau_f\tau_g.\) Then, \(\mu_{f+g}/\tau_{f+g}\leq (\mu_f\mu_g)/(\tau_f\tau_g)<2\) where the last inequality follows by Proposition \ref{prop:oldbounds}. 
\end{proof}
In the case of Proposition \ref{prop:newbounds1}, since \(g\) is a plane curve singularity Corollary \ref{cor:maximaltau} shows that in fact we have 
\[\frac{\mu_{f+g}}{\tau_{f+g}}=\frac{\mu_f\mu_g}{\tau_f\tau_g+\nu_1^{(f)}(\mu_g-\tau_g)}.\]
In order to find a family where the bound of Proposition \ref{prop:newbounds1} is asymptotically sharp one should deal with the uncomfortable term \(\nu_1^{(f)}(\mu_g-\tau_g).\) Moreover, it is reasonable to think that the bound of Proposition \ref{prop:newbounds1} should not be asymptotically sharp.

\printbibliography
\end{document}